\documentclass[preprint,12pt]{elsarticle}
\usepackage{amsthm,amsmath,amssymb}
\usepackage[colorlinks=true,citecolor=black,linkcolor=black,urlcolor=blue]{hyperref}
\usepackage{graphicx}
\usepackage{float}
\usepackage{mathtools}
\usepackage{epstopdf}
\usepackage{epsfig}
\usepackage{subfigure}

\theoremstyle{plain}
\newtheorem{theorem}{Theorem}
\newtheorem{lemma}[theorem]{Lemma}
\newtheorem{corollary}[theorem]{Corollary}
\newtheorem{proposition}[theorem]{Proposition}

\theoremstyle{definition}
\newtheorem{definition}[theorem]{Definition}

\newtheorem{conjecture}[theorem]{Conjecture}

\newtheorem{problem}[theorem]{Problem}

\theoremstyle{remark}
\newtheorem{remark}[theorem]{Remark}

\title{Twist monomials  of binary delta-matroids}
\author{Qi Yan\\
\small School of Mathematics\\[-0.8ex]
\small China University of Mining and Technology\\[-0.8ex]
\small P. R. China\\
Xian'an Jin\footnote{Corresponding author.}\\
\small School of Mathematical Sciences\\[-0.8ex]
\small Xiamen University\\[-0.8ex]
\small P. R. China\\
\small{\tt Email:qiyan@cumt.edu.cn; xajin@xmu.edu.cn}
}
\date{}

\begin{document}
\begin{abstract}
Recently, we introduced the twist polynomials of delta-matroids and gave a characterization of even normal binary delta-matroids whose twist polynomials have only one term and posed a problem: what would happen for odd binary delta-matroids?  In this paper,  we show that a normal binary delta-matroid whose twist polynomials have only one term if and only if each connected component of the intersection graph of the delta-matroid is either a complete graph of odd order or a  single vertex with a loop.
\end{abstract}
\maketitle

\section{Introduction}
The partial dual with respect to a subset $A$ of edges of a ribbon graph $G$ was introduced by Chmutov \cite{CG} in connection with the Jones-Kauffman and Bollob\'{a}s-Riordan polynomials. In 2020, Gross, Mansour and Tucker \cite{GMT} introduced the partial duality polynomial of a ribbon graph, the generating function enumerating partial duals by Euler-genus and proposed the following conjecture.
\begin{conjecture}[\cite{GMT}]\label{con1}
 There is no orientable ribbon graph having a non-constant partial duality polynomial with only one non-zero coefficient.
\end{conjecture}

In \cite{QYXJ}, we found an infinite family of counterexamples to this conjecture.  Essentially, these are the only counterexamples \cite{SCFV, QYXJ2}. Chumutov and Vignes-Tourneret \cite{SCFV} also showed that it would be interesting to know whether the partial duality polynomial and the related conjectures would make sence for general delta-matroids.
In \cite{QYXJ3}, we showed that partial duality polynomials have delta-matroid analogues. We introduced the twist polynomials of delta-matroids and discussed its basic properties for delta-matroids. We gave a characterization of even normal binary delta-matroids with one term twist polynomials and posed a problem:

\begin{problem}[\cite{QYXJ3}]
What would happen for odd binary delta-matroids  with only one term twist polynomials?
\end{problem}

In this paper we answer this problem and  the main result of this paper is a characterization of normal binary delta-matroids whose twist polynomials have only one term:

~

\noindent {\bf Main Theorem.}  Let $D=(E, \mathcal{F})$ be a normal binary delta-matroid. Then $^{\partial}w_{D}(z)=mz^k$ if and only if each connected component of $G_{D}$ is either a complete graph of odd order or a  single vertex with a loop.

\section{Preliminaries}

We give a brief review of delta-matroids and related terminologies, and refer the reader to \cite{AB1, CISR, CMNR, JO} for further details.

 A \emph{set system} is a pair $D=(E, \mathcal{F})$, where $E$ or $E(D)$, is a finite set,  called the \emph{ground set}, and $\mathcal{F}$ or $\mathcal{F}(D)$, is a collection of  subsets of $E$, called \emph{feasible sets}. $D$ is \emph{proper} if $\mathcal{F}\neq \emptyset$,  is \emph{trivial} if $E=\emptyset$, and is \emph{normal} if $\mathcal{F}$ contains the empty set.  $D$ is said to be \emph{even} if the cardinality of the sets in $\mathcal{F}$ all have the same parity. Otherwise, we call $D$ \emph{odd}.
 Bouchet \cite{AB1} introduced delta-matroids as follows.

\begin{definition}[\cite{AB1}]
A \emph{delta-matroid} is a proper set system $D=(E, \mathcal{F})$ for which  satisfies the symmetric exchange axiom: for all triples $(X, Y, u)$ with $X, Y \in \mathcal{F}$ and $u\in X\Delta Y$, there is a $v\in X\Delta Y$ (possibly  $v=u$ ) such that $X\Delta \{u, v\}\in \mathcal{F}$. 
\end{definition}

Here and below $\Delta$ denotes the symmetric difference operation on pairs of sets, $$X\Delta Y:=(X\cup Y)\backslash (X\cap Y).$$ And $|A|$ denotes the cardinality of a finite set $A$.

Let $D=(E, \mathcal{F})$ be a delta-matroid. If for any $F_{1}, F_{2}\in \mathcal{F}$, we have $|F_{1}|=|F_{2}|$. Then $D$ is said to be a \emph{matroid} and we refer to $\mathcal{F}$ as its \emph{bases}. If a delta-matroid forms a matroid $M$, then we usually denote $M$ by $(E, \mathcal{B})$. We say that the \emph{rank} of $M$, written $r(M)$, is equal to $|B|$ for any $B\in\mathcal{B}(M)$.

For a delta-matroid $D=(E, \mathcal{F})$, let $\mathcal{F}_{max}(D)$ and $\mathcal{F}_{min}(D)$ be the collection of
sets in $\mathcal{F}(D)$  that have the maximum and minimum cardinality among sets in $\mathcal{F}(D)$, respectively.
Bouchet \cite{AB2} showed that $D_{max}:=(E, \mathcal{F}_{max})$ and $D_{min}:=(E, \mathcal{F}_{min})$ are both matroids. $D_{min}$ is called the \emph{lower matroid}, and $D_{max}$ is called the \emph{upper matroid}. The  \emph{width} of $D$, denote by $w(D)$, is defined by $$w(D):=r(D_{max})-r(D_{min}).$$ We observe that $w(D)=r(D_{max})$, for a normal delta-matroid $D$.

In 1987, Bouchet \cite{AB1} introduced a fundamental operation on a delta-matroid called a twist. Given a delta-matroid $D=(E, \mathcal{F})$ and a subset $A$ of $E$, the \emph{twist} of $D$ with respect to $A$, denoted by $D*A$, is given by $$(E, \{A\Delta X: X\in \mathcal{F}\}).$$ The \emph{dual} of $D$, written $D^{*}$, is equal to $D*E$. Observe that the twist of a delta-matroid is a delta-matroid \cite{AB1}.

\begin{definition}[\cite{QYXJ3}]
The  \emph{twist polynomial} of any delta-matroid $D=(E, \mathcal{F})$ is the generating function
$$^{\partial}w_{D}(z):=\sum_{A\subseteq E}z^{w(D*A)}$$
that enumerates all twists of $D$ by width.
\end{definition}

In particular, a one term twist  polynomial is called a \emph{twist monomial}. Observe that analyzing the twist polynomials of all delta-matroids are equivalent to analyzing normal delta-matroids \cite{QYXJ3}. Consequently, it suffices to consider normal delta-matroids.

\begin{definition}[\cite{CMNR}]
For delta-matroids $D=(E, \mathcal{F})$ and $\widetilde{D}=(\widetilde{E}, \widetilde{\mathcal{F}})$ with $E\cap \widetilde{E}=\emptyset$, the  \emph{direct sum} of $D$ and $\widetilde{D}$, written $D\oplus \widetilde{D}$, is the delta-matroid defined as $$D\oplus \widetilde{D}:=(E\cup \widetilde{E}, \{F\cup \widetilde{F}: F\in \mathcal{F}~\text{and}~\widetilde{F}\in \widetilde{\mathcal{F}}\}).$$
\end{definition}

A delta-matroid is  \emph{disconnected} if it can be written as $D\oplus \widetilde{D}$ for some non-trivial delta-matroids $D$ and $\widetilde{D}$, and \emph{connected} otherwise.

Let $D=(E, \mathcal{F})$  be a delta-matroid. An element $e\in E$ is a \emph{coloop} if for each $F\in \mathcal{F}$ we have $e\in F$,
and it is a \emph{loop} if for any $F\in \mathcal{F}$ we have $e\notin F$.

\begin{definition}[\cite{CMNR}]
Let $D=(E, \mathcal{F})$ be a delta-matroid. Take $e\in E$. Then

\begin{description}
  \item[(1)] $e$ is a \emph{ribbon loop} if $e$ is a loop in $D_{min}$;
  \item[(2)] A ribbon loop $e$ is \emph{non-orientable} if $e$ is a ribbon loop in $D*e$ and is \emph{orientable} otherwise.
\end{description}
\end{definition}

Let $D=(E, \mathcal{F})$ be a delta-matroid and $e\in E$. Then $D$ \emph{delete} by $e$, denoted $D\backslash e$, is defined as $D\backslash e:=(E\backslash e, \mathcal{F}')$, where
\[\mathcal{F}':=\left\{\begin{array}{ll}
\{F: F\in \mathcal{F}, F\subseteq E\backslash e\}, & \text{if $e$ is not a coloop,}\\

\{F\backslash e: F\in \mathcal{F}\}, & \text{if $e$ is a coloop}.
\end{array}\right.\]
$D$ \emph{contract} by $e$, denoted $D/ e$, is defined as $D/ e:=(E\backslash e, \mathcal{F}'')$, where
\[\mathcal{F}'':=\left\{\begin{array}{ll}
\{F\backslash e: F\in \mathcal{F}, e\in F\}, & \text{if $e$ is not a loop,}\\

\mathcal{F}, & \text{if $e$ is a loop}.
\end{array}\right.\]
Note that $D^{*}/e=(D\setminus e)^{*}$ \cite{CMNR}.

Bouchet \cite{AB1} has shown that the order in which deletions are performed does not matter.
Let $D=(E, \mathcal{F})$  be a delta-matroid and $A\subseteq E$. We define $D\setminus A$ as the result of deleting every element of $A$ in any order.
The complement of $A\subseteq E$ is $A^c := E\setminus A$.
The \emph{restriction} of $D$ to $A$, written $D|_{A}$, is the set system $D\setminus A^{c}$. Throughout the paper, we will often omit the set brackets in the case of a single element set. For example, we write $D*e$ instead of $D*\{e\}$, or $D|_{e}$ instead of $D|_{\{e\}}$.

For a finite set $E$, let $C$ be a symmetric $|E|\times|E|$ matrix over $GF(2)$, with rows and columns indexed, in the same order, by the elements of $E$. Let $C[A]$ be the principal submatrix of $C$ induced by the set $A\subseteq E$. We define the set system $D(C)=(E, \mathcal{F})$ with $$\mathcal{F}:=\{A\subseteq E: C[A] \mbox{ is non-singular}\}.$$  By convention $C[\emptyset]$ is non-singular. Bouchet \cite{AB4} showed that $D(C)$ is a normal delta-matroid. A delta-matroid is said to be \emph{binary} if it has a twist that is isomorphic to $D(C)$ for some symmetric matrix $C$ over $GF(2)$.

In particular, if $D=(E, \mathcal{F})$ is a normal binary delta-matroid, then there exists a unique symmetric $|E|\times|E|$ matrix $C$ over $GF(2)$, whose rows and columns are labelled (in the same order) by the set $E$ such that $D=D(C)$. In fact, the matrix $C$
can be constructed  as follows \cite{AB3, Moff}:
\begin{description}
\item [(1)] Set $C_{v, v}=1$ if and only if $\{v\}\in \mathcal{F}$. This determines the diagonal entries of $C$;
\item [(2)] Set $C_{u,v}=1$ if and only if $\{u\}, \{v\}\in \mathcal{F}$ but $\{u, v\}\notin \mathcal{F}$, or $\{u, v\}\in \mathcal{F}$ but $\{u\}$ and $\{v\}$ are not both in $\mathcal{F}$. Then the feasible sets of size two determine the off-diagonal entries of $C$.
\end{description}
Observe that the construction above gives a way to  define the intersection graph of any normal binary delta-matroid $D$ with respect to the unique matrix $C$ of $D$. The \emph{intersection graph} \cite{Moff} $G_{D}$ of a normal binary delta-matroid $D$ is the graph with the vertex set $E$ and in which two vertices $u$ and $v$ of $G_{D}$ are adjacent if and only if $C_{u, v}=1$ and there is a loop at $v$ if and only if $C_{v, v}=1$. Note that $D$ is connected if and only if $G_{D}$ is connected.

\section{The Proof of the Main Theorem}
\begin{proposition}[\cite{QYXJ3}]\label{pro 2}
Let $D=(E, \mathcal{F})$ and $\widetilde{D}=(\widetilde{E}, \widetilde{\mathcal{F}})$ be two delta-matroids and $A\subseteq E$. Then
\begin{description}
\item [(1)] $^{\partial}w_{D}(z)=~^{\partial}w_{D*A}(z);$
\item [(2)] $^{\partial}w_{D\oplus \widetilde{D}}(z)=~^{\partial}w_{D}(z)~^{\partial}w_{\widetilde{D}}(z).$
\end{description}
\end{proposition}

\begin{lemma}[\cite{CMNR}]\label{lem 2}
Let $D=(E, \mathcal{F})$ be a delta-matroid with $r(D_{min})=r$ and suppose that $e$ is a non-orientable ribbon loop of $D$. Then a subset $F$ of $E-e$ is a basis of $D_{min}$ if and only if $F\cup e$ is a feasible set of $D$ with cardinality $r+1$.
\end{lemma}

\begin{lemma}\label{lem 6}
Let $D=(E, \mathcal{F})$ be a delta-matroid. If $e$ is a non-orientable ribbon loop and $f$ is a non-ribbon loop of $D$, then $f$ is a non-ribbon loop of $D/e$.
\end{lemma}

\begin{proof}
Since $e$ is a ribbon loop of $D$, it follows that $e$ is a loop in $D_{min}$. Then for any $F\in \mathcal{F}_{min}(D)$, $e\notin F$.
Furthermore, since $f$ is a non-ribbon loop of $D$, there exists $F'\in \mathcal{F}_{min}(D)$ such that $f\in F'$ and  $e\notin F'$.
Thus $F'\cup e\in \mathcal{F}(D)$ by Lemma \ref{lem 2}. We observe that $F'\in \mathcal{F}_{min}(D/e)$. Thus $f$ is not a loop in
$(D/e)_{min}$ and hence $f$ is a non-ribbon loop of $D/e$.
\end{proof}

\begin{lemma}[\cite{QYXJ3}]\label{lem 5}
Let $D=(E, \mathcal{F})$ be a normal delta-matroid and $A\subseteq E$. Then \[w(D*A)=w(D|_{A})+w(D|_{A^{c}}).\]
\end{lemma}

\begin{lemma} \label{lem 4}
Let $D=(E, \mathcal{F})$ be a normal delta-matroid, and let $e\in E$ with $w(D)=w(D*e)$.
\begin{description}
  \item[(1)] If $e$ is an orientable ribbon loop of $D$, then $e$ is a non-ribbon loop of $D^{*}$.
  \item[(2)] If $e$ is a non-orientable ribbon loop of $D$, then $e$ is a non-orientable ribbon loop of $D^{*}$.
\end{description}



\end{lemma}

\begin{proof}

{\bf (1)} Since $e$ is an orientable ribbon loop, we have $D|_{e}=(\{e\}, \{\emptyset\})$. Then $w(D|_{e})=0$. Note that $$w(D*e)=w(D|_{e})+w(D\backslash e)$$ by Lemma \ref{lem 5}. Since $w(D)=w(D*e)$, it follows that $w(D\backslash e)=w(D)$.
Thus $r((D\backslash e)_{max})=r(D_{max})$ and hence there exists $F\in \mathcal{F}_{max}(D)$ such that $e\notin F$. Then $F^{c}\in \mathcal{F}_{min}(D^*)$ and $e\in F^{c}$. Therefore $e$ is a non-ribbon loop of $D^{*}$.

  ~

\noindent {\bf (2)} Since $e$ is a non-orientable ribbon loop, it follows that $D|_{e}=(\{e\}, \{\emptyset, \{e\}\})$. Then $w(D|_{e})=1$ and hence $w(D\backslash e)=w(D)-1$. Thus $$r((D\backslash e)_{max})=r(D_{max})-1.$$ Then for any $X\in \mathcal{F}_{max}(D)$, $e\in X$ and there exists $Y\in \mathcal{F}(D)$ such that $|Y|=r(D_{max})-1$ and $e\notin Y$. Notice that this means that  for any $X'\in \mathcal{F}_{min}(D^*)$, $e\notin X'$ and there exists $Y'\in \mathcal{F}(D^*)$ such that $|Y'|=r({D^*}_{min})+1$ and $e\in Y$.
Thus $e$ is a ribbon loop of both $D^{*}$ and $D^{*}*e$. Hence $e$ is a non-orientable ribbon loop of $D^{*}$.
\end{proof}

\begin{theorem}\label{the 1}
Let $D=(E, \mathcal{F})$ be a connected odd normal binary delta-matroid. Then $^{\partial}w_{D}(z)=mz^k$ if and only if $D=(\{1\}, \{\emptyset, \{1\}\})$.
\end{theorem}

\begin{proof}
If $D=(\{1\}, \{\emptyset, \{1\}\})$, then $^{\partial}w_{D}(z)=2z$. Thus the sufficiency is verified. For necessity, we claim that $|E|=1$.
Suppose not. Then we consider two claims as follows.

\begin{description}
  \item[Claim 1.] For any $e, f\in E$, $D|_{\{e, f\}}\neq (\{e, f\}, \{\varnothing, \{e\}, \{e, f\}\})$.

Suppose that Claim 1 is not true. Then  there exist $e, f\in E$ such that $$D|_{\{e, f\}}= (\{e, f\}, \{\varnothing, \{e\}, \{e, f\}\}).$$ It is easy to verify that $e$ is a non-orientable ribbon loop and $f$ is  an orientable ribbon loop of $D$.
Since $^{\partial}w_{D}(z)=mz^k$, it follows that $$w(D)=w(D*e)=w(D*f).$$ Then $e$ is a non-orientable ribbon loop and $f$ is  a non-ribbon loop of $D^{*}$ by Lemma \ref{lem 4}. Thus $f$ is a non-ribbon loop of $D^{*}/e$ by Lemma \ref{lem 6}. Note that $D^{*}/e=(D\setminus e)^{*}$ and hence $f$  is a non-ribbon loop of  $(D\setminus e)^{*}$. Then there exists $F\in \mathcal{F}_{min}((D\setminus e)^{*})$ such that $f\in F$. Let $F':=(E\backslash e)\backslash F$,  then $F'\in \mathcal{F}_{max}(D\setminus e)$ and $e, f\notin F'$. Therefore $F'\in \mathcal{F}(D)$ by the definition of $D\backslash e$. For any $X\in \mathcal{F}_{max}(D)$, we observe that $e\in X$.  Otherwise, $$r(D*e_{max})=r(D_{max})+1.$$ Since $\{e\}\in \mathcal{F}(D)$, we have $\emptyset\in \mathcal{F}(D*e)$, that is, $r(D*e_{min})=0$. Then $w(D*e)=w(D)+1$, this contradicts $^{\partial}w_{D}(z)=mz^k$.  Thus $$r((D\backslash e)_{max})\leq r(D_{max})-1.$$ Furthermore, we observe that there exists $Y\in \mathcal{F}(D)$ such that $e\notin Y$ and $|Y|=r(D_{max})-1$. Otherwise, $$r(D*e_{max})=r(D_{max})-1.$$ Then $w(D*e)=w(D)-1$, this also contradicts $^{\partial}w_{D}(z)=mz^k$. Thus $Y\in D\backslash e$ and hence $$r((D\backslash e)_{max})=r(D_{max})-1.$$ We obtain $|F'|=r(D_{max})-1$. Since $\emptyset, F'\cup \{e, f\}\in \mathcal{F}(D*\{e, f\})$, it follows that $$w(D*\{e, f\})\geq |F'\cup \{e, f\}|=r(D_{max})+1=w(D)+1,$$ this contradicts $^{\partial}w_{D}(z)=mz^k$. Hence Claim 1 is proved.

  \item[Claim 2.]  For any $e, f\in E$, $D|_{\{e, f\}}\neq (\{e, f\}, \{\varnothing, \{e\}, \{f\}\})$ .

Suppose that Claim 2 is not true. Then  there exist $e, f\in E$ such that $$D|_{\{e, f\}}= (\{e, f\}, \{\varnothing, \{e\}, \{f\}\}).$$ It is easily seen that $$D*e|_{\{e, f\}}=(\{e, f\}, \{\varnothing, \{e\}, \{e, f\}\}).$$ Note that
$^{\partial}w_{D*e}(z)=~^{\partial}w_{D}(z)=mz^k$ by Proposition \ref{pro 2} (1), this contradicts Claim 1. Then Claim 2 follows.
\end{description}

Since $D$ is an odd normal binary delta-matroid, we know that $D=D(C)$ for some symmetric matrix $C$ over $GF(2)$ and there exists a non-orientable ribbon loop $e$. As $D$ is connected and $|E|\geq 2$, there exists $f\in E$ such that
\[C[\{e, f\}] =
\bordermatrix{
& e & f   \cr
e & 1   & 1     \cr
f  & 1   & 0     \cr
}\]
 or
 \[C[\{e, f\}] =
\bordermatrix{
& e & f   \cr
e & 1   & 1     \cr
f  & 1   & 1     \cr
}.\]
Then $$D|_{\{e, f\}}= (\{e, f\}, \{\varnothing, \{e\}, \{e, f\}\})$$ or $$D|_{\{e, f\}}= (\{e, f\}, \{\varnothing, \{e\}, \{f\}\}),$$ this contradicts Claim 1 or 2. Thus $|E|=1$. Since $D$ is an odd delta-matroid, it follows that $D=(\{1\}, \{\emptyset, \{1\}\})$.
\end{proof}

\begin{corollary}\label{cor 1}
Let $D=(E, \mathcal{F})$ be a connected odd normal binary delta-matroid. Then $^{\partial}w_{D}(z)=mz^k$ if and only if $G_{D}$ is a single vertex with a loop.
\end{corollary}
\begin{proof}
Since the intersection graph of $D=(\{1\}, \{\emptyset, \{1\}\})$ is a single vertex with a loop,  the result follows immediately from Theorem \ref{the 1}.
\end{proof}

\begin{proposition}[\cite{QYXJ3}]\label{main-3}
Let $D=(E, \mathcal{F})$ be a connected even normal binary delta-matroid. Then $^{\partial}w_{D}(z)=mz^k$ if and only if $G_{D}$ is a complete graph of odd order.
\end{proposition}

\begin{remark}
The proof of the Main Theorem is straightforward by Propositions \ref{pro 2} (2), \ref{main-3} and Corollary \ref{cor 1}.
\end{remark}

\section*{Acknowledgements}
This work is supported by NSFC (Nos. 12171402, 12101600) and the Fundamental Research Funds for the Central Universities (Nos. 20720190062, 2021QN1037).


\end{document}